\documentclass[12pt]{article}

\usepackage{amssymb}

\newcommand{\normed}[1]{\| #1 \|}

\newcommand{\qed}{$\;\;\;\Box$}
\newenvironment{proof}{\par\smallbreak{\sl Proof.~}}
{\unskip\nobreak\hfill \qed \par\medbreak}

\newcounter{claim}
\renewcommand{\theclaim}{\arabic{claim}}
{\par\medskip\par}

{\qed\par\smallbreak}
\newcommand{\hide}[1]{}

\renewcommand{\H}{{\mbox{H}}}
\newcommand{\Con}{{\mbox{C}}}
\newcommand{\D}{{\cal D}}

\newcommand{\N}{{\mathbb N}}
\newcommand{\R}{{\mathbb R}}
\newcommand{\C}{{\mathbb C}}
\newcommand{\Z}{{\mathbb Z}}

\newcommand{\beq}{\begin{equation}}
\newcommand{\ee}{\end{equation}}

\renewcommand{\d}{\partial}

\newtheorem{thm}{Theorem}
\newtheorem{lemma}[thm]{Lemma}

\newcommand{\al}{\alpha}
\newcommand{\be}{\beta}
\newcommand{\ga}{\gamma}

\newcommand{\eps}{\varepsilon}
\newcommand{\vphi}{\varphi}

\newcommand{\dist}{\mathop{\rm dist}}
\title{
$C^1$-smoothness of Nemytskii operators on Sobolev-type spaces
of periodic functions
}

\newcounter{thesame}
\author{
I.~Kmit\thanks{supported by a Humboldt Research Fellowship}
\\
{\small
Institute for Applied Problems of Mechanics and Mathematics,}\\
{\small Ukrainian Academy of Sciences}\\
{\small Naukova St.\ 3b,
79060 Lviv,
Ukraine}
\\
{\small   E-mail:
{\tt kmit@informatik.hu-berlin.de}}
}

\date{}

\begin{document}

\maketitle

\begin{abstract}
\noindent
We consider a class of Nemytskii superposition operators that covers
the nonlinear part of traveling wave models from laser dynamics, population dynamics,
 and chemical kinetics.
Our main result is the $C^1$-continuity property of these operators over
Sobolev-type spaces of periodic functions.
\end{abstract}

smoothness


\section{Motivation and main result}\label{sec:intro}
Development of a bifurcation theory for hyperbolic PDEs encounters
significant difficulties caused by the fact that hyperbolic
operators have worse regularity properties in comparison to ODEs
and parabolic PDEs. Such a theory has to cover
one- and multi-parameter bifurcations  (both local and global),
stability of bifurcating solutions, and periodic
synchronizations. For hyperbolic problems all these topics currently remain
challenging research directions. In each of them, investigation of
smoothness properties of Nemytskii superposition operators plays
an important role.

Not losing potential applicability to the aforementioned topics, here we
consider Nemytskii operators in the context of the
traveling wave models from laser dynamics \cite{LiRadRe,Rad,RadWu}. The
models describe the dynamics of multisection semiconductor lasers. They
include a semilinear first-order one-dimensional hyperbolic system.

As an additional source of motivation, note that some problems of
population dynamics~\cite{Hillen,hiha,horst,Lutscher},
chemical kinetics \cite{akr,AkrBelZel,Sli,Zel1,Zel2}, and
kinetic gas dynamics \cite{conner,illner1,illner2} have the same
hyperbolic operator. Thus, our analysis applies to those problems as well,
even when they have a different type of boundary conditions.

In the case of the traveling wave models, we deal with periodic-Dirichlet problems
and our overall goal is to provide a bifurcation analysis for them. The
basic idea is to apply techniques based on the Implicit Function Theorem
in Banach spaces and the Lyapunov-Schmidt reduction
(see, e.g., \cite{ChowHale,Ki}).
The first problem to solve on this way is to establish the Fredholm solvability of
the corresponding linearized problems, what is done in \cite{KmRe1,KmRe2}.
To make the linearization procedure correct and to solve the so-called ``range''
equation (obtained after a Lyapunov-Schmidt reduction) via Implicit Function Theorem,
 we would need  appropriate smoothness properties
of the Nemytskii superposition operators with respect to the function spaces
used in \cite{KmRe1,KmRe2}. The results obtained in this paper are sufficient
to achieve this goal.

Due to the great importance of Nemytskii operators in the theory of nonlinear equations,
their smoothness properties in different function spaces were extensively studied
(see, e.g., \cite{appel}). Here we involve into consideration new function spaces
important for solving nonlinear hyperbolic PDEs.

To state our main result, let us introduce the function spaces we are working with:
For $\ga\ge 0$ we denote by $W^{\ga}$ the vector Banach space of all locally integrable functions
$u: [0,1]\times\R\to\R^n$ such that
$u(x,t)=u\left(x,t+2\pi\right)$ for almost all $x \in (0,1)$ and $t\in\R$
and that
\begin{equation}\label{eq:W}
\|u\|_{W^{\ga}}^2=\sum\limits_{s\in\Z}(1+s^2)^{\gamma}
\int\limits_0^1\left\|\int\limits_0^{2\pi}
u(x,t)e^{-ist}\,dt\right\|^2\,dx<\infty.
\end{equation}
Here and throughout $\|\cdot\|$ is the Hermitian norm in $\C^n$.
In other words, $W^\ga$  is the anisotropic Sobolev space of all  measurable functions
$u: [0,1]\times\R\to\R^n$ such that $u(x,t)=u\left(x,t+2\pi\right)$ for almost all  $x \in (0,1)$ and $t\in\R$
and that the distributional partial derivatives of $u$ with respect to $t$ up to the order $\ga$ are  locally
quadratically integrable.
Furthermore, given $a \in L^\infty\left((0,1);\R^n\right)$ with $\mbox{ess inf } |a_j|>0$
for all $j\le n$, we introduce  the function spaces
$$
V^{\gamma} = \Bigl\{u\in W^{\gamma}:\, \d_xu\in W^{\ga-1},
\left[\d_tu_j+a_j\d_xu_j\right]_{j=1}^n\in W^{\gamma}\Bigr\}
$$
endowed with the norms
\begin{equation}\label{eq:norm1}
\|u\|_{V^{\gamma}}^2=\|u\|_{W^{\gamma}}^2
+\left\|\left[\d_tu_j+a_j\d_xu_j\right]_{j=1}^n\right\|_{W^{\gamma}}^2.
\end{equation}
In the notation $V^\ga$ we drop the dependence of this space on $a$. It should be stressed that our results
hold true for each $a$.
Note that the space $V^\ga$ is larger than the space of  all $u \in W^\ga$ with
$\partial_t u \in  W^\ga$ and $\partial_x u \in  W^\ga$.

We will focus on the pair of function spaces $(V^2,W^2)$, for which we prove
our main result given by  Theorem~\ref{thm:Nemytskii}. It is important that
$V^2$ is embedded into the  algebra
of (continuous) functions with pointwise multiplication
(see Assertion $(ii)$ of Lemma~\ref{lem:embedding}
and the embedding (\ref{eq:emb2}) below). This will allow us to use
pointwise nonlinearities for the description of our Nemytskii operators.

Given a function $f(x,u): (0,1)\times\R^n\to\R$  defined for almost all
$x\in(0,1)$ and all  $y\in\R^n$, let
\beq\label{eq:F}
\left[F(u)\right](x,t)=f(x,u(x,t)).
\ee
We will show that $F$ is a $C^1$-smooth superposition operator from $V^2$ into $W^2$.

For the sake of technical simplicity and without loss of generality  we
will suppose that $n=1$.

\begin{thm}\label{thm:Nemytskii}
Suppose that  $f(\cdot,\cdot)\in\allowbreak L^\infty\left(0,1;\Con^{4}[-M,M]\right)$
for each $M>0$. Then   $F(u) \in C^1(V^2,W^2)$.
\end{thm}

It should be emphasized here that, by physical reasons, the function $f$
can have discontinuities with respect to the first argument, and
the assumption of the theorem covers such cases.

Note also that
under additional regularity assumptions on $f$, we can extend Theorem~\ref{thm:Nemytskii}
to any desired smoothness
of the operator $F$ and to the  pair of spaces $(V^\ga,W^\ga)$ for any
integer $\ga\ge 2$.

\section{Properties of the used function spaces}

As usual,  by $H^1\left(0,1\right)$ we denote the Sobolev space of all functions
$u \in L^2\left(0,1\right)$ such that the weak derivative $u'$ belongs to $L^2\left(0,1\right)$.
The norm in $H^1(0,1)$ is defined by
$$
\|u\|^2_{H^1\left(0,1\right)}=\sum_{j=0}^1\int_0^1|u^{(j)}(x)|^2 dx.
$$
Similarly, by $H^1\left((0,1) \times (0,2\pi)\right)$ we denote the Sobolev space of all functions
$u \in L^2((0,1) \allowbreak \times(0,2\pi))$ such that for every multiindex $\al=(\al_1,\al_2)\in\N_0^2$
with $|\al|\le 1$, the weak partial derivative $D^\al u$ belongs to $L^2\left((0,1)\times(0,2\pi)\right)$.
The norm in $H^1\left((0,1) \times (0,2\pi)\right)$ is given by
$$
\|u\|^2_{H^1\left((0,1) \times (0,2\pi)\right)}=\sum_{|\al|\le 1}\int_0^1\int_0^{2\pi}|D^\al u(x,t)|^2\, dxdt.
$$
Moreover, by $H^1\left((0,2\pi);H^1\left(0,1\right)\right)$  we denote
the abstract Sobolev space of all locally quadratically Bochner integrable maps
$u:(0,2\pi) \to H^1\left(0,1\right)$ such that the
 distributional derivative
$u'$ is also locally quadratically Bochner integrable, with the norm
$$
\|u\|^2_{H^1\left((0,2\pi);H^1\left(0,1\right)\right)}
=\sum_{j=0}^1\int_0^{2\pi}\|u^{(j)}(t)\|_{H^1\left(0,1\right)}^2dt.
$$
Note that the space $H^1\left((0,2\pi);H^1(0,1)\right)$ is smaller than the classical Sobolev space
$H^1\left((0,1) \times (0,2\pi)\right)$, and we have the continuous embeddings
\begin{eqnarray}
\label{eq:emb1}
H^1\left((0,1) \times (0,2\pi)\right) & \hookrightarrow & L^p\left((0,1) \times (0,2\pi)\right)
\mbox{ for all } p \in [2,\infty),\\
\label{eq:emb2}
H^1\left((0,2\pi);H^1(0,1)\right) & \hookrightarrow &
C\left([0,1] \times [0,2\pi]\right),
\end{eqnarray}
see~\cite[Theorem 5.4]{AD}.

We now establish some properties of the function spaces
$V^1$ and $V^2$ introduced in Section~\ref{sec:intro},
which are needed for proving Theorem~\ref{thm:Nemytskii}.

\begin{lemma}\label{lem:embedding}
We have the following continuous embeddings:

(i) $V^1 \hookrightarrow H^1\left((0,1) \times (0,2\pi)\right)$;

(ii) $V^2 \hookrightarrow H^1\left((0,2\pi); H^1\left(0,1\right)\right)$.
\end{lemma}

\begin{proof}
Notice the continuous embedding
\beq\label{eq:e}
V^\ga\hookrightarrow W^\ga\hookrightarrow W^{\ga-1},\quad \ga\ge 1,
\ee
that is a straightforward consequence of the definitions of the spaces
$V^\ga$ and $W^\ga$.

(i) Take $u\in V^1$. Then $u\in W^1$ and, therefore,  $\d_tu \in  W^0$ with
\beq\label{eq:est1}
\|\d_tu\|_{W^0}^2\le \|u\|_{W^0}^2+\|\d_tu\|_{W^0}^2=\|u\|_{W^1}^2\le \|u\|_{V^1}^2.
\ee
Moreover, by the definition of $V^1$, we have
$\d_tu+a\d_xu\in W^1$. On the account of the embedding (\ref{eq:e}),
\beq\label{eq:est2}
\|\d_tu+a\d_xu\|_{W^0}^2\le \|\d_tu+a\d_xu\|_{W^1}^2+
\|u\|_{W^1}^2= \|u\|_{V^1}^2.
\ee
By triangle inequality
\beq\label{eq:est3}
\|a\d_xu\|_{W^0}^2-\|\d_tu\|_{W^0}^2\le \|\d_tu+a\d_xu\|_{W^0}^2.
\ee
Since
$a \in L^\infty(0,1)$ with $\mbox{ess inf } |a|>0$,
it follows by (\ref{eq:est1})--(\ref{eq:est3}), that
$$
\|\d_xu \|_{W^0}\le  c\|u\|_{V^1},
$$
where the constant $c$ does not depend on $u$.
Therefore
$$
\|u\|_{W^0}+\|\d_xu\|_{W^0}+\|\d_tu\|_{W^0}\le
(2+c)\|u\|_{V^1}.
$$
To finish  the proof of this part, it remains to note that $W^0=L^2\left((0,1) \times (0,2\pi)\right)$.

(ii) We proceed similarly: Take $u\in V^2$. Then $u\in W^2$, and we have $u$ as well as  $\d_tu$ and $\d_xu$ in $W^1$.
Moreover, $\|u\|_{W^1}\le \|u\|_{W^2}\le\|u\|_{V^2}$ and $\|\d_tu\|_{W^1}\le \|u\|_{W^2}\le\|u\|_{V^2}$.
This implies that
$\|\d_xu\|_{W^1}\le c\|u\|_{V^2}$, where the constant $c$ does not depend on $u$. Claim (ii) readily follows
from these estimates.
\end{proof}

The following fact is similar to the density result for Sobolev spaces (see~\cite[Section III]{AD})
and proved by the same method.

\begin{lemma}\label{lem:dense}
The subspace $C^\infty\cap V^2$ is dense in $V^2$.
\end{lemma}

\begin{proof}
Set $\Pi=(0,1)\times(0,2\pi)$.
By periodicity, speaking of a function in $V^2$, we can assume its restriction
to $\overline\Pi$. We will use this convention in the course of the proof
of the lemma.

Let $\vphi$ be a non-negative $C^\infty(\R^2)$-function that
vanishes outside a unit disk and satisfies the condition
$\int\vphi(x)\,dx=1$. Take $u\in V^2$
and consider its regularization defined by
$$
u_\eps(x,t)=\frac{1}{\eps^2}\int_\Pi
u(\xi,\tau)\vphi\left(\frac{x-\xi}{\eps},
\frac{t-\tau}{\eps}\right)\,d\xi d\tau
$$
for $\eps<\dist\left((x,t),\d\Pi\right)$. Due to the properties of
the convolutions, for any strict subdomain $\Pi^\prime\subset\Pi$
it holds $\d_t^\al u_\eps\to \d_t^\al u$ and
$\d_t^\be\left[\d_x u_\eps\right]\to  \d_t^\be\left[\d_x u\right]$
in $L^2\left(\Pi^\prime\right)$ as $\eps\to 0$ for $\al=0,1,2$ and $\be=0,1$
(see \cite[Section III]{AD} for details). This implies, in particular,
that $v_\eps\to v$ in $L^2\left(\Pi^\prime\right)$ as
$\eps\to 0$, where $v=\d_tu+a(x)\d_xu$ and
$v_\eps=\d_tu_\eps+a(x)\d_xu_\eps$.
Now we intend to prove
that $u_\eps\to u$ in $V^2$ on $\Pi^\prime$ as $\eps\to 0$. It
suffices to show that $\d_t^\al v_\eps\to \d_t^\al v$ in $L^2\left(\Pi^\prime\right)$ as
$\eps\to 0$ for $\al=1,2$.  Fix
$\eps_0<\dist\left(\Pi^\prime,\d\Pi\right)$ and consider
$\eps<\eps_0$. Then for any $\psi\in C_0^\infty\left(\Pi^\prime\right)$ we have
\begin{eqnarray*}
\lefteqn{
\int_{\Pi^\prime}\left(\d_tu_\eps(x,t)+a(x)\d_xu_\eps(x,t)\right)
\d_t^\al\psi(x,t)\,dxdt}\\&&
=\frac{1}{\eps^2}\int_{\R^2}\int_{\R^2}
\left[\d_tu+a\d_xu\right](x-\xi,t-\tau)
\vphi\left(\frac{\xi}{\eps},\frac{\tau}{\eps}\right)\d_t^\al\psi(x,t)\,d\xi
d\tau dxdt\\&& =\frac{(-1)^\al}{\eps^2}\int_{\R^2}\int_{\R^2}
\d_t^\al\left[\d_tu+a\d_xu\right](x-\xi,t-\tau)
\vphi\left(\frac{\xi}{\eps},\frac{\tau}{\eps}\right)\psi(x,t)\,d\xi
d\tau dxdt
\\&& =(-1)^\al\int_{\Pi^\prime}
\d_t^\al\left(v\right)_\eps(x,t) \psi(x,t)\,dxdt.
\end{eqnarray*}
Therefore, $\d_t^\al\left(v_\eps\right)(x,t)=\left(\d_t^\al
v\right)_\eps(x,t)$ in the sense of distributions on $\Pi^\prime$.
Since $\d_t^\al v\in L^2(\Pi)$ for $\al=1,2$,
$$
\lim\limits_{\eps\to 0}\left\|\d_t^\al v_\eps-\d_t^\al
v\right\|_{L^2\left(\Pi^\prime\right)}= \lim\limits_{\eps\to
0}\left\|\left(\d_t^\al v\right)_\eps-\d_t^\al
v\right\|_{L^2\left(\Pi^\prime\right)}=0,
$$
as desired.

Consider now the following locally finite open covering of $\Pi$:
\begin{eqnarray*}
\Pi_1&=&\left\{(x,t)\in\Pi\,:\,\dist\left((x,t),\d\Pi\right)>\frac{1}{2}\right\},\\
\Pi_j&=&\left\{(x,t)\in\Pi\,:\,\frac{1}{j+1}<\dist\left((x,t),\d\Pi\right)<\frac{1}{j-1}\right\},
\quad j\ge 2.
\end{eqnarray*}
Let $\eta_1,\eta_2,\dots$ be a partition of unity subordinate to
the covering $\left\{\Pi_{j+1}\setminus\Pi_{j-1}\right\}$. Then,
given $j\ge 1$, the product $\eta_ju$ is in $V^2$ and has
support contained in $\Pi_j$. Consider now the mollification
$(\eta_ju)_\eps$. Given $\eps_0>0$, we can choose
a sequence $\eps_j$ such that
$$
\eps_j<\dist\left(\Pi_{j+1},\d\Pi_{j+3}\right) \mbox{ and }
\|(\eta_ju)_{\eps_j}-\eta_ju\|_{V^2}\le\frac{\eps_0}{2^{j+1}}.
$$
Let $w=\sum_{j=1}^\infty(\eta_ju)_{\eps_j}$. It follows from the
definition of the partition of unity that at each $x\in\Pi$ only
finitely many terms in the sum are nonzero. Since each term is
smooth, this implies $w\in C^\infty(\Pi)$. Moreover, using the
triangle inequality, we have
$$
\|w-u\|_{V_n^2}\le\sum_{j=1}^{n+2}
\|(\eta_ju)_{\eps_j}-\eta_ju\|_{V_n^2}\le\sum_{j=1}^\infty\eps_02^{-j}=\eps_0,
$$
where $\|\cdot\|_{V_n^2}$ is defined by (\ref{eq:norm1}) with the
integral over
$$
\Pi_{1/n}=\left\{(x,t)\in\Pi\,:\,\dist\left((x,t),\d\Pi\right)>\frac{1}{n}\right\}
$$
in place of the integral over $\Pi$.
This yields
$$
\|w-u\|_{V^2}=\sup\limits_{n\ge 1}\|w-u\|_{V_n^2}\le\eps_0.
$$
Since $\eps_0>0$ is arbitrary, the set
$\sum_{j=1}^n(\eta_ju)_{\eps_j}$, $n\ge 3$, is the desired dense set from $C^\infty\cap V^2$.

\end{proof}

\section{$\Con^1$-smoothness of the Nemytskii operator from
$V^2$ into $W^2$ (proof of Theorem~\ref{thm:Nemytskii})}

We split the proof into two lemmas.

\begin{lemma}\label{lem:cont}
The superposition operator $F$
given by the formula~(\ref{eq:F}) maps $V^2$ into $W^2$.
\end{lemma}

\begin{proof}
For any function $u\in V^2$, denote by
$F'(u)$ and $F''(u)$ the superposition operators by putting, for
almost all $x\in(0,1)$,
$$
\left[F'(u)\right](x,t)=\left(\d_uf\right)(x,u(x,t)),
$$
$$
\left[F''(u)\right](x,t)=\left(\d^2_{u}f\right)(x,u(x,t)).
$$
 As $V^2\hookrightarrow C\left([0,1]\times[0,2\pi]\right)$ continuously
(see Lemma~\ref{lem:embedding} (ii) and the embedding (\ref{eq:emb2})),
we can identify any $u\in V^2$  with a
uniformly continuous and $2\pi$-periodic in $t$ function on $[0,1]\times\R$.
Furthermore, we have the inequality
\beq\label{eq:C0}
\|u\|_{C\left([0,1]\times[0,2\pi]\right)}\le C_0\|u\|_{V^2}\quad\mbox{ for all }\quad
u\in V^2,
\ee
the constant $C_0$ being independent of $u$. Combining this with the smoothness
assumptions on $f$, we conclude that, given
$u\in V^2$, the functions $[F(u)](x,t)$, $[F'(u)](x,t)$, and $[F''(u)](x,t)$
belong to $L^\infty\left((0,1)\times(0,2\pi)\right)$.

{\it Claim 1.  $F(u)$ maps $V^2$ into $W^1$.}
 Fix an arbitrary $u\in V^2$,  set
\begin{equation}\label{eq:K}
K=\normed{u}_{C([0,1]\times[0,2\pi])},
\end{equation}
and consider $(u^m)_{m\in\Z}$ to be a sequence in $C^\infty\cap V^2$
converging to $u$ in $V^2$. By (\ref{eq:C0}), we have this convergence also  in
$C\left([0,1]\times[0,2\pi]\right)$. For almost all $x\in(0,1)$ and all $t\in\R$ we have
\beq\label{eq:n5}
[\d_tF(u^m)](x,t)=[F'(u^m)](x,t)\d_tu^m(x,t).
\ee
Let us show that
\beq\label{eq:n6}
F'(u^m)\d_tu^m\to F'(u)\d_tu
\mbox{ in } L^2\left((0,1)\times(0,2\pi)\right) \mbox{ as } m\to\infty.
\ee
Indeed,
\begin{eqnarray}
\label{eq:Du}
\lefteqn{
\int\limits_0^1\int\limits_0^{2\pi}\left|F'(u^m)\d_tu^m-
F'(u)\d_tu\right|^2\,dx\,dt}
\nonumber\\
&&\le 2\int\limits_0^1\int\limits_0^{2\pi}\left|F'(u^m)-F'(u)
\right|^2
\left|\d_tu^m\right|^2\,dx\,dt
\nonumber\\&&
+2\int\limits_0^1\int\limits_0^{2\pi}\left|F'(u)\right|^2
\left|\d_tu^m-\d_tu\right|^2\,dx\,dt
\\&&
\le 2\int\limits_0^1\int\limits_0^{2\pi}\left|
\int\limits_0^1(\d_u^2f)(x,\sigma u^m+(1-\sigma)u)\, d\sigma\right|^2
\left|u^m-u\right|^2
\left|\d_tu^m\right|^2\,dx\,dt\nonumber\\
&&+2\int\limits_0^1\int\limits_0^{2\pi}\left|(\d_uf)(x,u)\right|^2
\left|\d_tu^m-\d_tu\right|^2\,dx\,dt\nonumber\\
&&
\le 2\left\|u^m-u\right\|_{C([0,1]\times[0,2\pi])}^2
\left\|
\d_u^2f\right\|_{L^\infty\left((0,1)\times
\left(-3K;3K\right)\right)}^2
\left\|\d_tu^m\right\|_{W^0}^2\nonumber\\
&&+2\left\|\d_uf\right\|_{L^\infty\left((0,1)\times
\left(-K;K\right)\right)}^2
\left\|\d_tu^m-\d_tu\right\|_{W^0}^2.
\end{eqnarray}
The latter inequality is true for all sufficiently large $m\in\N$.
Since $(u^m)_{m\in\N}$ converges to $u$ in $V^2$ and
$
V^2\hookrightarrow L^2(0,1;\H_{2\pi}^1(0,2\pi)),
$
the sequence $(\d_tu^m)_{m\in\N}$ is bounded in $L^2\left((0,1)\times(0,2\pi)\right)$ and
 converges to $\d_tu$ in  $L^2\left((0,1)\times(0,2\pi)\right)$. This shows the convergence
(\ref{eq:n6}). It follows by
H\"older's inequality that for any $\vphi\in\D\left((0,1)\times(0,2\pi)\right)$
\begin{eqnarray}
\label{eq:n7}
\lefteqn{
\int\limits_0^1\int\limits_0^{2\pi}\left(F(u)\d_t\vphi+F'(u)\d_tu\vphi\right)\,dxdt}\\
&&=\lim\limits_{m\to\infty}\left[
\int\limits_0^1\int\limits_0^{2\pi}\left(F(u^m)\d_t\vphi+F'(u^m)\d_tu^m\vphi\right)\,dxdt\right].\nonumber
\end{eqnarray}
By~(\ref{eq:n5}),
the expression under the limit sign is equal to zero.
Hence (\ref{eq:n7}) implies
$$
\int\limits_0^1\int\limits_0^{2\pi}\left(F(u)\d_t\vphi+F'(u)\d_tu\vphi\right)
\,dxdt=0
$$
for any $\vphi\in\D\left((0,1)\times(0,2\pi)\right)$.
This means that $F(u)$ has a weak partial derivative in $t$ given by the formula
$$
\d_tF(u)=F'(u)\d_tu.
$$
Recall that $[F'(u)](x,t)\in L^\infty\left((0,1)\times(0,2\pi)\right)$ and
$\d_tu\in L^2\left((0,1)\times(0,2\pi)\right)$. It is immediate that
$[\d_tF(u)](x,t)\in L^2\left((0,1)\times(0,2\pi)\right)$
and therefore $[F(u)](x,t)\in W^1$.
Since $u\in V^2$ is arbitrary, the desired assertion is therewith proved.

{\it Claim 2.  $F(u)$ maps $V^2$ into $W^2$.}
As above, fix an arbitrary $u\in V^2$ and choose $(u^m)_{m\in\Z}$ as in Claim~1.
Similarly to the proof of Claim~1, one can show the convergence
\beq\label{eq:n9}
\begin{array}{cc}
F''(u^m)\left(\d_tu^m\right)^2+F'(u^m)\d_t^2u^m
\to
F''(u)\left(\d_tu\right)^2+F'(u)\d_t^2u\\
\mbox{ in } L^2\left((0,1)\times(0,2\pi)\right)
\mbox{ as } m\to\infty
\end{array}
\ee
and that
\beq\label{eq::11primed}
\d_t^2F(u)=F''(u)\left(\d_tu\right)^2+F'(u)\d_t^2u.
\ee
The only difference appearing here concerns the estimation of the following integral:
\begin{eqnarray}
\label{eq:D2u}
\lefteqn{
\int\limits_0^1\int\limits_0^{2\pi}\left|F''(u^m)\left(\d_tu^m\right)^2
-F''(u)\left(\d_tu\right)^2\right|^2\,dx\,dt}\nonumber\\
&&\le 2\int\limits_0^1\int\limits_0^{2\pi}\left|(\d_u^2f)(x,u^m)-
(\d_u^2f)(x,u)\right|^2\left|\d_tu^m\right|^4\,dx\,dt\nonumber\\
&&+2\int\limits_0^1\int\limits_0^{2\pi}\left|(\d_u^2f)(x,u)\right|^2
\left|\left(\d_tu^m\right)^2-\left(\d_tu\right)^2\right|^2\,dx\,dt\nonumber\\
&&\le 2\int\limits_0^1\int\limits_0^{2\pi}\left|\int\limits_0^1
(\d_u^3 f)
\left(x,\sigma u^m+(1-\sigma)u\right)\,d\sigma\right|^2
\left|u^m-u\right|^2\left|\d_tu^m\right|^4\,dx\,dt\nonumber\\
&&+2\left\|
\d_u^2f\right\|_{L^\infty\left((0,1)\times
\left(-K;K\right)\right)}^2\nonumber\\
&&
\times\int\limits_0^1
\left\|\d_tu^m(x,\cdot)-\d_tu(x,\cdot)\right\|_{L^\infty(0,2\pi)}^2\,dx
\int\limits_0^{2\pi}\left\|\d_tu^m(\cdot,t)+\d_tu(\cdot,t)\right\|_{L^\infty(0,1)}^2\,dt
\nonumber\\
&&
\le
2\left\|\d_u^3f\right\|_{L^\infty\left((0,1)\times
\left(-3K;3K\right)\right)}^2
\left\|u^m-u\right\|_{C}^2\left\|\d_tu^m\right\|_{L^4}^4\nonumber\\
&&+2\left\|
\d_u^2f\right\|_{L^\infty\left((0,1)\times
\left(-K;K\right)\right)}^2\nonumber
\int\limits_0^1
\left\|\d_tu^m(x,\cdot)-\d_tu(x,\cdot)\right\|_{L^\infty(0,2\pi)}^2\,dx\\
&&\times
\int\limits_0^{2\pi}\left\|\d_tu^m(\cdot,t)+\d_tu(\cdot,t)\right\|_{L^\infty(0,1)}^2\,dt,
\end{eqnarray}
where the constant $K$ is defined by the formula (\ref{eq:K}).
The right hand side tends to zero by Lemma~\ref{lem:embedding},
 the embedding (\ref{eq:emb1}),  and the embedding
$$
V^2\hookrightarrow W^2\hookrightarrow L^2\left(
0,1;C^1[0,2\pi]\right).
$$
Turning back to (\ref{eq::11primed}), we obtain
$[\d_t^2F(u)](x,t)\in L^2\left((0,1)\times(0,2\pi)\right)$.
Hence $[F(u)](x,t)\in W^2$ as desired.
\end{proof}

\begin{lemma}\label{lem:C1}
The mapping $u\in V^2\to F(u)\in W^2$  is $C^1$-smooth and
for all $u,v\in V^2$ it holds
\beq\label{eq:n20}
 \left[F'(u)v\right](x,t)=(\d_uf)(x,u(x,t))v(x,t).
\ee
\end{lemma}

\begin{proof}
We now prefer to work with  the following norm in $W^2$:
\begin{equation}\label{eq:norm}
\|w\|^2_{W^2}=\|\d_t^2 w\|^2_{W^0}.
\end{equation}
Note that it is equivalent to the $W^2$-norm  introduced by (\ref{eq:W}).

To prove the continuity of the mapping $u\in V^2\to F(u)\in W^2$,
fix an arbitrary $u\in V^2$.
On the account of the expression (\ref{eq::11primed}) for $\d^2_tF(u)$
 and the estimates (\ref{eq:Du}) and (\ref{eq:D2u})
with $u^m$ replaced by $u+v$, we derive the following inequality
for all $v\in V^2$  with $\|v\|_{V^2}\le K/C_0$, where the constant $C_0$
is fixed to satisfy (\ref{eq:C0}) and $K$ is determined by (\ref{eq:K}):
\begin{eqnarray*}
\lefteqn{
\frac12\normed{\d^2_tF(u+v)(x,t)-\d^2_tF(u)(x,t)}^2_{W^0}\le
}\\[2mm]&&
\normed{\d^3_uf}^2_{L^\infty\left((0,1)\times(-3K,3K)\right)}
\normed{\d_t(u+v)}^2_{L^2(0,1;L^\infty(0,2\pi))}
\\[2mm]&&
\times
\normed{\d_t(u+v)}^2_{L^2(0,2\pi;L^\infty(0,1))}\normed{v}^2_{C([0,1]\times[0,2\pi])}
\\[2mm]&&
+\normed{\d^2_uf}^2_{L^\infty\left((0,1)\times(-K,K)\right)}
\normed{\d_t(2u+v)}^2_{L^2(0,2\pi;L^\infty(0,1))}\normed{\d_tv}^2_{L^2(0,1;L^\infty(0,2\pi))}
\\[2mm]&&
+\normed{\d^2_uf}^2_{L^\infty\left((0,1)\times(-3K,3K)\right)}
\normed{\d^2_t(u+v)}^2_{W^0}\normed{v}^2_{C([0,1]\times[0,2\pi])}
\\[2mm]&&
+\normed{\d_uf}^2_{L^\infty\left((0,1)\times(-K,K)\right)}
\normed{\d^2_tv}^2_{W^0}
\le C\normed{v}^2_{V^2},
\end{eqnarray*}
the constant $C$ being dependent on  $f$ and $u$, but not  on $v$. We conclude that
$$
\normed{\d^2_tF(u+v)(x,t)-\d^2_tF(u)(x,t)}^2_{W^0}=O(\normed{v}_{V^2}^2)
$$
as $\normed{v}_{V^2}\to 0$. The continuity of $F$ is therefore proved.

Let us now show that the operator $u\to F(u)$ is continuously differentiable.
Fix $u\in V^2$ and introduce the bounded linear operator
$G\,:\,V^2\to W^2$ defined by the formula
$$
[G(u)v](x,t)=(\d_uf)(x,u(x,t))v(x,t).
$$
From the smoothness assumptions on $f$ and the proof of Lemma~\ref{lem:cont}
it follows that $(\d_uf)(x,u(x,t))\in W^2$. Since  $V^2\hookrightarrow W^2$
continuously, $W^2$ is an algebra of functions, and $v\in V^2$,
the correctness of the definition of the operator $G$ is straightforward.

Our next concern is to show that $F$ is differentiable in $u$ and that
$F'(u)=G(u)$. Similarly to the above, fix $u\in V^2$ and consider
 $w\in V^2$ with
$\|w\|_{V^2}\le K/C_0$, where $C_0$ is a certain constant satisfying (\ref{eq:C0}) and
$K$ is specified by (\ref{eq:K}). It follows  by (\ref{eq:C0}) that $\normed{w}_{C([0,1]\times[0,2\pi])}\le K$.
The desired assertion now follows from the following estimate:
\begin{eqnarray*}
\lefteqn{
\left\|F(u+w)(x,t)-F(u)(x,t)-[G(u)w](x,t)\right\|_{W^2}}\\
&&=\left\|f(x,u+w)-f(x,u)-(\d_uf)(x,u)w\right\|_{W^2}\\
&&=\left\|w\int\limits_0^1
\left[
(\d_uf)(x,u+\sigma w)-(\d_uf)(x,u)
\right]\,d\sigma\right\|_{W^2}\\
&&=\left\|w^2\int\limits_0^1\int\limits_0^1\sigma
(\d_u^2f)(x,u+\sigma\sigma_1w)\,d\sigma d\sigma_1\right\|_{W^2}\\
&&=\left\|\d_t^2\left[w^2\int\limits_0^1\int\limits_0^1\sigma
(\d_u^2f)(x,u+\sigma\sigma_1w)\,d\sigma d\sigma_1\right]\right\|_{W^0}\\
&&=\bigg\|2(w\d_t^2w+(\d_tw)^2)\int\limits_0^1\int\limits_0^1\sigma
(\d_u^2f)(x,u+\sigma\sigma_1w)\,d\sigma d\sigma_1\\
&&+w^2\int\limits_0^1\int\limits_0^1\sigma
(\d_u^4f)(x,u+\sigma\sigma_1w)
\left[\d_t u+\sigma\sigma_1\d_t w\right]^2\,d\sigma d\sigma_1\\
&&+w^2\int\limits_0^1\int\limits_0^1\sigma
(\d_u^3f)(x,u+\sigma\sigma_1w)
\left[\d_t^2 u+\sigma\sigma_1\d_t^2 w\right]\,d\sigma d\sigma_1\\
&&+2w\d_tw\int\limits_0^1\int\limits_0^1\sigma
(\d_u^3f)(x,u+\sigma\sigma_1w)
\left[\d_t u+\sigma\sigma_1\d_t w\right]\,d\sigma d\sigma_1\bigg\|_{W^0}\\
&&
\le
4\left(\|w\|_{L^4}^2+\|\d_tw\|_{L^4}^2+\|w\|_{V^2}\|w\|_{C}\right)\normed{f(\cdot,\cdot)}_{L^\infty\left((0,1),C^4(-3K,3K)\right)}\\
&&
\times\bigg(
1+\normed{u}_{W^2}+\normed{\d_tu}_{L^4\left((0,1)\times(0,2\pi)\right)}^2
+\normed{w}_{W^2}
+\normed{w}_{L^4\left((0,1)\times(0,2\pi)\right)}^2
\bigg).
\end{eqnarray*}
In the last inequality we again used Lemma~\ref{lem:embedding} and the embedding (\ref{eq:emb1}).
The continuous differentiability of $F$ is proved, which completes the proof of the lemma.
\end{proof}

\end{document}